\documentclass[letterpaper]{amsart}
\usepackage{amsmath}
\usepackage{amssymb}
\usepackage{graphicx}
\usepackage{hyperref}
\usepackage{mathrsfs} % for mathscr font

\DeclareMathOperator{\red}{Red}
\DeclareMathOperator{\rpp}{rpp}

\DeclareMathOperator{\defect}{Defect}
\newcommand{\Push}{\mathscr{P}}
\newcommand{\Bump}{\mathscr{B}}
\newcommand{\Delete}{\mathscr{D}}
\newcommand{\PD}{\mathscr{P}\!\!\!\mathscr{D}}
\newcommand{\BD}{\mathscr{B}\!\mathscr{D}}
\newcommand{\IB}{\mathscr{I}\!\!\!\mathscr{B}}

\newtheorem{theorem}{Theorem}[section]
\newtheorem{definition}[theorem]{Definition}
\newtheorem{lemma}[theorem]{Lemma}
\newtheorem{corollary}[theorem]{Corollary}
\newtheorem{proposition}[theorem]{Proposition}
\newtheorem{algorithm}[theorem]{Algorithm}
\newcommand{\bfabarbar}{\overline{\overline{\mathbf{a}}}}
\newcommand{\bfa}{\mathbf{a}}

\newlength{\figwidth}
\begin{document}
\title[Macdonald's distribution on Reduced Words]{A Markov growth process for Macdonald's distribution on reduced words}
\author{Benjamin Young}

\begin{abstract}
We give an algorithmic-bijective proof of Macdonald's reduced word identity in the theory of Schubert polynomials, in the special case where the permutation is dominant.  Our bijection uses a novel application of David Little's generalized bumping algorithm.  We also describe a Markov growth process for an associated probability distribution on reduced words.  Our growth process can be implemented efficiently on a computer and allows for fast sampling of reduced words.  We also discuss various partial generalizations and links to Little's work on the RSK algorithm.
\end{abstract}

\maketitle

%=================================================================================
\section{Introduction and notation}
\label{sec:intro}

The theory of Schubert polynomials has many beautiful identities which lack bijective proofs.  In this paper, we will give an algorithmic bijection which proves a special case of Macdonald's identity.  In order to state this identity, we must first review several standard definitions from the literature.  

\subsection{Permutations, reduced words and wiring diagrams}
\label{subsec:permutations}

    Let $S_n$ denote the symmetric group on $n$ elements, and let $\pi \in S_n$ be a permutation.  We will usually represent permutations in \emph{one--line notation} -- that is, by listing $\pi(1), \pi(2), \ldots, \pi(n)$, omitting the commas when giving explicit examples.  The \emph{permutation matrix} of $\pi$ is the zero-one matrix $M$ with ones in position $(\pi(i), i)$ for $1 \leq i \leq n$, and zeroes elsewhere.

    Let $s_i$ denote the \emph{elementary transposition} $(i, i+1)$ for $1 \leq i < n$.  Any sequence of positive integers $(a_i)_{1 \leq i \leq k}$ in the range $1 \leq a_i \leq n$ is called a \emph{word}; moreover, if $\bfa = (a_1, a_2, \ldots, a_k)$ is a word such that $\pi = s_{a_{1}}s_{a_2}\cdots s_{a_k}$, and if $k$ is equal to the number of inversions of $\pi$, then we say $\bfa$ is a \emph{reduced word} or \emph{reduced decomposition} for $\pi$.  
    
    One can represent the reduced word $\bfa=(a_1, \ldots, a_k)$ for $\pi$ by a \emph{wiring diagram}, as follows.  For $0 \leq t \leq k$, define the \emph{partial permutations}
\[\pi_t = \prod_{i=1}^t s_{a_i},\]
and observe that $\pi_0$ is the identity and $\pi_k = \pi$.  The $i$th \emph{wire} of $\bfa$ is defined to be the piecewise linear path joining the points $(i, \pi_t(i))$, $0 \leq t \leq k$.  We will consistently use matrix coordinates rather than cartesian coordinates in this paper, so that in an ordered pair $(i,j)$, $i$ refers to the row number (measured from the top) whereas $j$ refers to the column number measured from the left.  The \emph{wiring diagram} is the union of all of the $n$ wires.  
For each $t \geq 1$, observe that between column $t-1$ and $t$, precisely two wires $w_1, w_2$ intersect, froming an X in the wiring diagram; this configuration is called a \emph{crossing}.  One can identify a crossing either by its position $t$, or by the unordered pair $\{w_1,w_2\}$ of wires which are involved (since the word is reduced); the crossing is said to be at \emph{height $a_t$}.  As such, one can alternately construct a wiring diagram by first drawing the crossings at height $a_t$, and then joining them with horizontal line segments to form the wires.

For example: if $\pi$ is the permutation 4213, then $\pi$ has four inversions, so its reduced words are of length four.  One of these reduced words is $\bfa=(3,1,2,1)$.  The permutation matrix for $\pi$, and the wiring diagram for $\bfa$, are, respectively, 
\begin{center}
$
\displaystyle
\left[
\begin{matrix}
0 & 0 & 1 & 0 \\
0 & 1 & 0 & 0 \\
0 & 0 & 0 & 1 \\
1 & 0 & 0 & 0
\end{matrix}
\right]$
and
\raisebox{-1.5 \baselineskip}{ \includegraphics[height=0.6in]{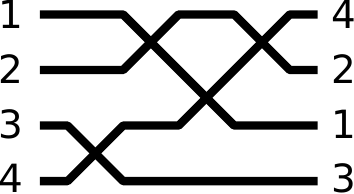}.  }
\end{center}
We would say, that for example, that the crossing of wires 1 and 4 occurs at position 3 and height 2.

\subsection{Dominant permutations and Macdonald weight}
We now define several fundamental concepts from the theory of Schubert polynomials, starting with a classical definition of Rothe~\cite{rothe} and used subsequently by many authors~\cite{lascoux-schutzenberger, manivel, macdonald, lascoux-book}.
\begin{definition}
    Let $\pi$ be a permutation with permutation matrix $M$.  The \emph{Rothe diagram} of $\pi$ is the collection of cells $(i,j)$ such that $M_{i',j} = 0$ for all $i' \leq i$ , and $M_{i,j'} = 0$ for all $j' \leq j$.  
\end{definition}
In other words, the Rothe diagram of $\pi$ is the set of cells which remain after striking out all entries directly below or directly to the right of each 1 in the permutation matrix of $\pi$.  It is customary to draw the Rothe diagram of $\pi$ as a collection of unit squares in the plane.  It is easy to check that $(i,j)$ is an inversion of $\pi$ if and only if $(\pi(j), i)$ is a cell in the Rothe diagram for $\pi$.  The crossings in a reduced word also represent inversions, so one way to represent a reduced word is to label the cells of the Rothe diagram with the numbers $1$ through $k$ where the corresponding crossing is to be found in the reduced word.  These are the so-called ``labelled circle diagrams'' of~\cite{edelman-greene, fomin-greene-reiner-shimozono, saga}; we will not need to use them in this paper.

Here are two examples of Rothe diagrams, for the permutations 4213 (left), whose Rothe diagram consists of the cells $\{(1,1), (1,2), (2,1), (2,2)\}$, and 2413 (right), whose Rothe diagram consists of the cells $\{(1,1), (1,2), (3,2)\}$.  In both cases the Rothe diagrams consist of the unshaded cells.

\begin{center}
\includegraphics[height=0.75in]{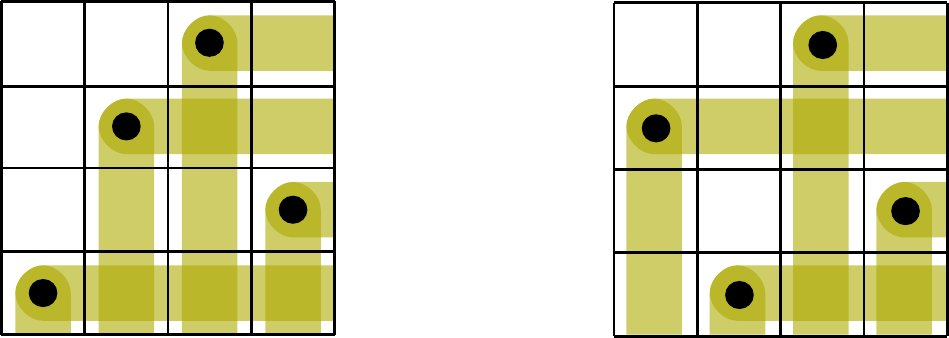}
\end{center}

\begin{definition}
    The permutation $\pi$ is \emph{dominant} if the Rothe diagram of $\pi$ is a Young diagram - that is, if the Rothe diagram of $\pi$ is either empty or if it contains only one connected component of cells in the upper left hand corner of the permutation matrix.
\end{definition}

For example, the longest element in $S_n$ is the permutation $n,n-1, \ldots, 3,2,1$ in one-line notation.  It is dominant; its diagram is the staircase Young diagram with $n-1$ stairs.  In the example above, 4213 is a dominant permutation, whereas 2413 is not.  It is clear that dominant permutations avoid the pattern $132$; in fact, dominance is equivalent to 132-avoidance~\cite{macdonald}.

The object of study in this paper is the following weight function. 
\begin{definition}  
\label{definition:Macdonald weight}
Let $\bfa = (a_1, a_2, \ldots)$ be a word.  Define
$\mu(\bfa) = \prod_t  a_t.$  We call $\mu(\bfa)$ the \emph{Macdonald weight} of $\bfa$.
\end{definition} 

Generally, we will take $\bfa$ to be a reduced word for the dominant permutation $\pi \in S_n$.  The Macdonald-weighted count of reduced words for such a $\pi$ has a particularly simple expression, due to Macdonald~\cite{macdonald}.
\begin{theorem}\cite[Equation (6.11)]{macdonald}  
\label{thm:macdonald}
If $\pi$ is a dominant permutation whose Rothe diagram is $\lambda \vdash k$, then 
\begin{equation}
\label{eqn:macdonald}
\sum_{\bfa \in \red(\pi)} \mu(\bfa) = k!.
\end{equation}
\end{theorem}

\subsection{Results}
Our main result is a bijective proof of Theorem~\ref{thm:macdonald}.  The bijection is given in Definition~\ref{def:the big bijection}, and the proof of bijectivity given immediately follows it.  

Fomin-Kirillov~\cite{fomin-kirillov} mention that there is a ``complicated'' bijective proof of this identity which has not since appeared in print.  Our bijection is surely not complicated; it is in fact an \emph{algorithmic} bijection, given by iterated use of the insert-bump maps $\IB$ described in Definition~\ref{def:bump_delete and insert_bump}.  These, in turn, are a novel application of Little's bumping algorithm~\cite{little1}.  Our bijection interprets the left side of Equation~\ref{eqn:macdonald} as the number of maximal-length paths in a certain ranked, multiple-edged directed graph $\Lambda_T$, which represents the outcomes of perfoming the $\IB$ maps (see Section~\ref{sec:bijection}).  The parameter $T$ is a standard Young tableau of shape $\lambda$, chosen arbitrarily (!).  The nodes in $\Lambda_T$ correspond to reduced words, and the edges correspond to outcomes of $\IB$.  Moreover, it is evident that $\Lambda_T$ has outdegree $n+1$ at rank $n$, which is enough to establish the identity.  

This bijection allows us to solve a second problem, suggested to the author by Alexander Holroyd: find an efficient algorithm for randomly generating a reduced word $\bfa$ for a dominant permutation $\pi$ with probability proportional to $\mu(\bfa)$.  If the length of $\pi$ is $k$ (so that reduced words for $\pi$ have $k$ inversions), then the constant of proportionality is $k!$.  

To sample from $\mu$, we perform a simple random walk in the graph $\Lambda_T$.  When performing this random walk, we start at the node corresponding to the empty word, and add one crossing at a time using the $\IB$ map of Definition~\ref{def:bump_delete and insert_bump}.  At each step, the result is $\mu$-distributed.  The insertion-bumping process is illustrated in Figure~\ref{fig:one step of the shuffle}, and all possible outcomes of several steps of the growth rule are shown in Figure~\ref{fig:macdonald shuffle graph}.  
The rule for adding crossings is a random Markov step.  It is also \emph{local} in both space and time: to write down a word with $k$ crossings, one must first generate a word with $k-1$ crossings, but no data other than this length $k-1$ word need be retained.  Moreover, only part of the word changes, and that by a small amount.  That is, the word \emph{grows} slowly.  We propose the term \emph{Markov growth process} for such rules.  These processes occur in many other fields of mathematics.  They are sometimes called \emph{building schemes}~\cite{winkler}.  In the particular case where the object being grown in is a perfect matching on a planar graph, they are sometimes called \emph{domino shuffling algorithms}~\cite{EKLP, borodin-gorin, nordenstam-young}. 

\subsection{Generalizations and literature review}

The theorem which we prove bijectively has several generalizations and extensions in the literature which currently lack bijective proofs.  

In fact, Macdonald proves considerably more than Theorem~\ref{thm:macdonald} in \cite[Equation (6.11)]{macdonald}, which is stated for arbitrary permutations $\pi$.  In this more general setting, the right hand side becomes $\mathfrak{S}_{\pi}(1)$: the number of terms in the Schubert polynomial associated to $\pi$.  This formula is commonly called \emph{Macdonald's formula} in the literature on Schubert polynomials.  Schubert polynomials are combinatorial objects which encode the intersection theory of the flag variety.  These polynomials were discovered in by Lascoux-Schutzenberger~\cite{lascoux-schutzenberger}, and have been actively studied and generalized by many mathematicians over the next forty years.  For introductions to Schubert polynomials, see ~\cite{macdonald, manivel, lascoux-book}.  However, for the immediate purpose of reading this paper, it is unnecessary for the reader to be familiar with the theory of Schubert polynomials.  This is essentially because in our setting, when $\pi$ is a dominant permutation, the number of terms in the Schubert polynomial is equal to one~\cite{macdonald}.  

\begin{figure}
\caption{One step of the Markov growth process.  A crossing is inserted above the marked point, and then a \emph{Little bump} is performed at the new crossing.  
\label{fig:one step of the shuffle}}
\begin{center}
\includegraphics[width=\figwidth]{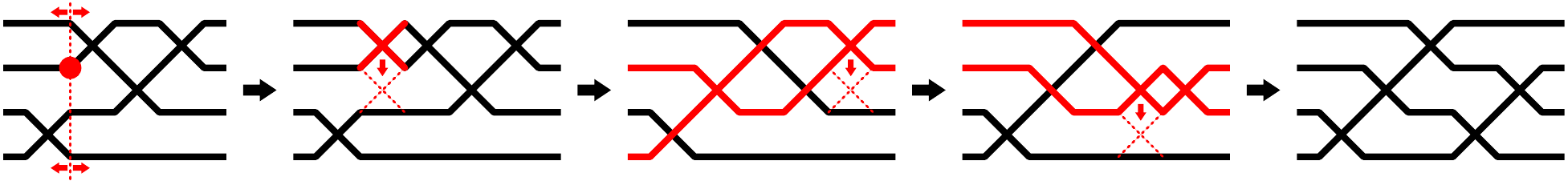}
\end{center}
\end{figure}

A curious extension of Theorem~\ref{thm:macdonald} was published in~\cite[Theorem 1.1]{fomin-kirillov}: 
\begin{proposition} 
\label{prop:fk}
If $\pi$ is a dominant permutation whose Rothe diagram is $\lambda \vdash k$, and $x \in \mathbb{N}$, then
\[
\sum_{\bfa = (a_1, a_2, \ldots) \in \red \pi} (x+a_1)(x+a_2) \cdots = k!\rpp(\lambda, x),
\]
where $\rpp(\lambda,x)$ denotes the number of reverse plane partitions of shape $\lambda$ and entries in the range $[0, x]$.
\end{proposition}
Proposition~\ref{prop:fk} is, in fact, a \emph{consequence} of Macdonald's formula, and a result of Wachs~\cite{wachs} which interprets the Schubert polynomial $\mathfrak{S}_{\pi}$ of a vexillary permutation as a flagged Schur function.  Schubert calculus, unfortuately, yields no hints as to how to make this result bijective.  
Our methods do yield an unweighted bijective interpretation for the left-hand side of Proposition~\ref{prop:fk}: it is the number of maximal-length paths in a certain ranked, multiple-edged directed graph (see Section~\ref{sec:bijection}).
  However, it is not as straightforward to interpret the right-hand side in a similar way; nor does simple random walk in this lattice generate such chains uniformly.  

Bijective proofs of the general Macdonald's formula, and of the non-$q$-analogue results of Fomin-Kirillov~\cite{fomin-kirillov} will appear in a forthcoming paper.
Work on the $q$-analogues of these results is ongoing.

\begin{figure}
\caption{The first few levels of the graph $\Lambda_{T}$, showing all possible trajectories of the markov growth process.  The diagrams on the left are dominant permutations $\pi$ and their diagrams $\lambda$.  The numbers beside each wiring diagram $\bfa$ are $\mu(\bfa)$; they coincide with $\#\{\text{paths to }\bfa\}$.   Random walk produces $\bfa$ with probability proportional to $\mu(\bfa)$.  Image Credit: Kristin Potter, \url{https://casit.uoregon.edu/}
\label{fig:macdonald shuffle graph}}
\begin{center}
\includegraphics[width=\figwidth]{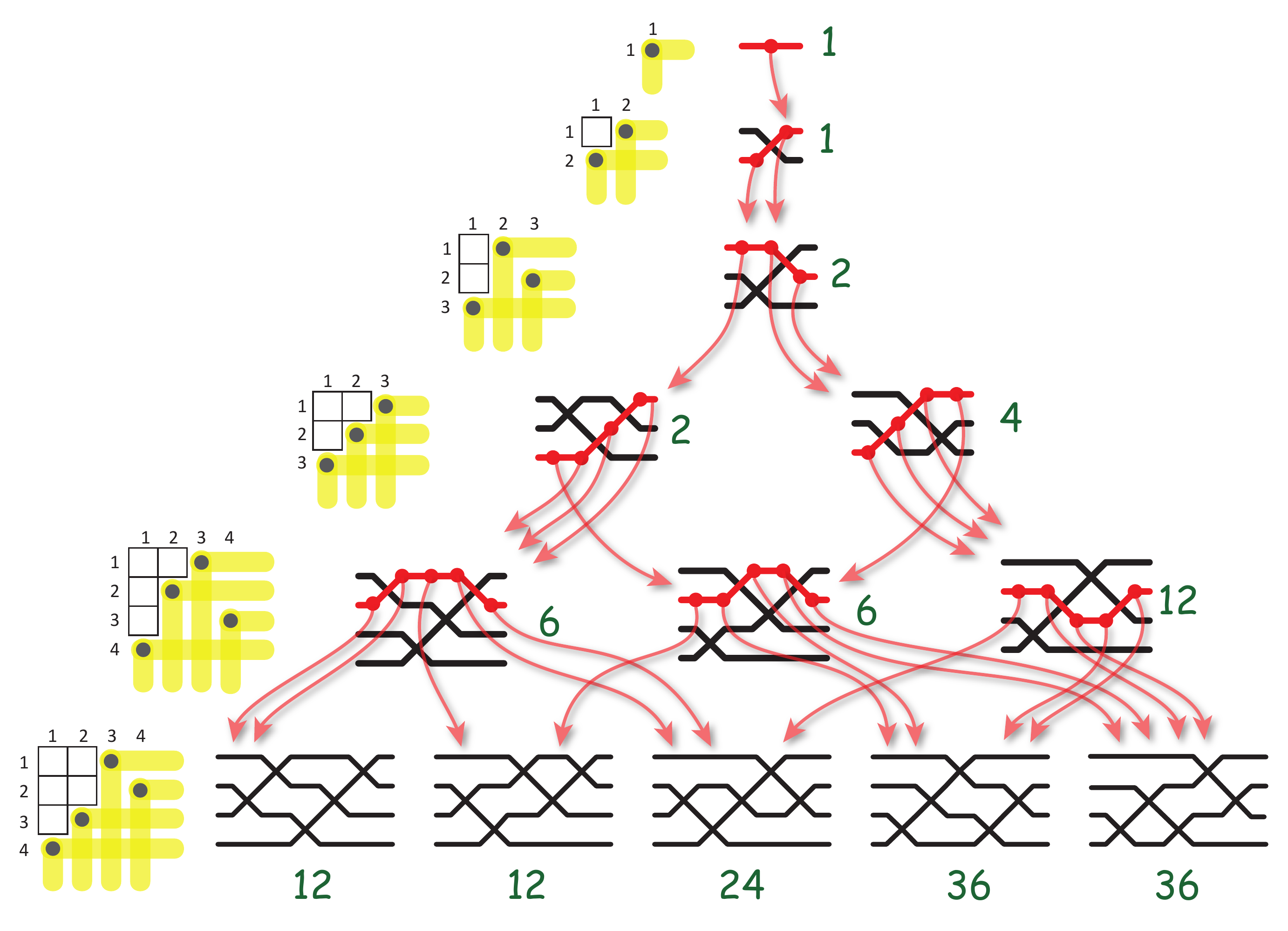}
\end{center}
\end{figure}

The author would like to thank 
Sara Billey, 
Vadim Gorin,
Zachary Hamaker, 
Alexander Holroyd, 
Greta Panova, 
Miklos Racz,
Jeff Remmel,
Alex Rozinov, 
Dan Romik,
Andrew Wilson
and
David Wilson, 
for many helpful conversations, as well as extending special thanks to Alexander Holroyd for the help in devising an efficient implementation of the Little bump, which was used to create the picture in Figure~\ref{fig:big wiring diagram}.  Computational work was done in part using SAGE~\cite{sage}.

%=================================================================================
\section{Definitions}
\label{sec:definitions}

We begin by recalling the definition for Little's bumping algorithm~\cite{little1} and establishing some notation for it.

\begin{definition}
\label{definition:push}
Let $\bfa = (a_1,a_2,\ldots, a_k)$ be a reduced word.  We define the \emph{Push up}, \emph{Push down}, and \emph{Deletion} of $\bfa$ at $t$, respectively, to be
\begin{align*}
\Push^{\uparrow}_{t}\bfa &= (a_1, \ldots, a_{t-1}, a_t - 1, a_{t+1}, \ldots, a_k),\\
\Push^{\downarrow}_{t}\bfa &= (a_1, \ldots, a_{t-1}, a_t + 1, a_{t+1}, \ldots, a_k),\\
\Delete_{t}\bfa &= (a_1, \ldots, a_{t-1}, a_{t+1}, \ldots, a_k).
\end{align*}
\end{definition}

\begin{definition}
Let $\bfa$ be a word.  If $\Delete_t\bfa$ is reduced, we say that $\bfa$ is \emph{nearly reduced at $t$.}
\end{definition}

The term ``nearly reduced'' was coined by Lam~\cite{lam}, who uses ``$t$-marked nearly reduced''.  
Words that are nearly reduced at $t$ may or may not also be reduced; however, every word $\bfa$ is nearly reduced at some index $t$.  For instance, Little~\cite{little1} observes that if $\bfa$ is a reduced word for the permutation $\pi$, then there is a canonical location $t$ where $\bfa$ is nearly reduced; $t$ is determined by the Lascoux-Schutzenberger tree of $\pi$~\cite{lascoux-schutzenberger, little1, saga}.  More obviously, any reduced word $\bfa$ of length $k$ is nearly reduced at 1 and also at $k$.

In order to define the Little bump, we need the following lemma, which to our knowledge first appeared in~\cite[Lemma 4]{little1}, and was later generalized to arbitrary Coxeter systems~\cite[Lemma 21]{lam}.  
\begin{lemma}
\label{lem:nearly_reduced}
If $\bfa$ is not reduced, but is nearly reduced at $t$, then $\bfa$ is nearly reduced at exactly one other position $t' \neq t$.
\end{lemma}

\begin{definition}
Using the notation of Lemma~\ref{lem:nearly_reduced}, we say that $t'$ \emph{forms a removable defect with} $t$ in $\bfa$, and write $\defect_t(\bfa) = t'$.
\end{definition}

The following is essentially Little's generalized bumping algorithm, defined in~\cite{little1}.  

\begin{algorithm}[Little Bumping Algorithm]
\label{algorithm:little bump}
Input: a word $\bfa'$ which is nearly reduced at $t_0$, and a direction $d \in \{\uparrow, \downarrow\}$.  Define $\Bump^{d}_{t_0}(\bfa')$ as follows:
\begin{enumerate}
\item Initialize $\bfa\leftarrow \bfa', t \leftarrow t_0$.
\item $\bfa \leftarrow \Push^{d}_{t}\bfa$.
\item $t \leftarrow \defect_t(\bfa)$.
\item If $\bfa$ is reduced, return $\bfa$.  Otherwise go to step 2.
\end{enumerate}
\end{algorithm}

The only significant difference between Little's ``more general bijection'' $\theta_r$ and our map $\Bump^{\uparrow}_t$, other than the indexing, is that $\theta_r$ shifts the entire word down (by applying $\prod_t\Push^{\downarrow}_t$, in our terminology) if a crossing is pushed onto the zero line, whereas our $\Bump^{\uparrow}_t$ map does not.  Rather, we simply introduce a new wire marked 0.  That is, we admit reduced words for permutations of the the points $(0,1, \ldots, n)$, though as we shall see this is for convenience only.  For the moment, observe that permutations which do not fix the point 0 necessarily have a crossing at height 0, and thus the Macdonald weight of any reduced word for such a permutation is 0.

The following definitions are taken from~\cite[Equations (6)-(9)]{little1}, modified in order to take the above difference into account.  
\begin{definition}
\label{definition:domain and range index sets for little bump}
Let $\tau_{i,j}$ denote the transposition $(i,j)$, and let $\ell(\pi)$ denote the number of inversions of the permutation $\pi$.  Define
\begin{align*}
I(\pi,r) &= \{0 \leq i < r\;|\; \ell(\pi \tau_{i,r}) = \ell(\pi) + 1\},\\
S(\pi,r) &= \{r < s  \;|\; \ell(\pi \tau_{r,s}) = \ell(\pi) + 1\},\\
\Phi(\pi,r) &= \{\pi \tau_{i,r}\;|\;i \in I(\pi,r)\}, \\
\Psi(\pi,r) &= \{\pi \tau_{r,s}\;|\;s \in S(\pi,r)\}. 
\end{align*}
\end{definition}

We will need several standard properties of the bumping algorithm.

\begin{proposition}(Properties of the Little Bump)
\label{proposition:little bump properties}
\begin{enumerate}
\item No crossing in $\bfa$ is moved more than once in Algorithm~\ref{algorithm:little bump}.
\item Algorithm~\ref{algorithm:little bump} terminates, and thus $\Bump^d_t \bfa$ is well-defined.
\item If $\bfa$ has a descent at $j$, then so does $\Bump^d_{t}(\bfa)$ for all positions $t$ where $a$ is nearly reduced.
\item If $s \in S(\pi,r)$ and $t(r,s)$ denotes the location of the crossing of wires $r$ and $s$, then 
$\Bump^{\downarrow}_{t(r,s)}$ is a bijection between the sets
\[\bigcup_{\rho \in \Psi(\pi \tau_{r,s},r)}\red(\rho) \longrightarrow \bigcup_{\rho \in \Phi(\pi \tau_{r,s}, r)} \red(\rho),\] 
\item Let $r$ be an arbitrary wire, and let $s \in S(\pi \tau_{r,s}, r)$.  Let $t_0$ denote the crossing of wires $r$ and $s$.  Calculate $\BD_{t_0}^d$ with Algorithm~\ref{algorithm:little bump}.  Then after every instance of step 3, $\Delete_t \bfa$ is a reduced word for $\pi \tau_{r,s}$.
\end{enumerate}
\end{proposition}
\begin{proof}
Properties 1 through 4 are the lemmas of~\cite{little1}; they are stated and proven there.   Property 3 is, in fact, a consequence of property 1. 

Property 5 is implicit in~\cite{little1}, and is stated and proven in~\cite{little2}.  The proof of property 5 given there is pretty and short, so we reproduce it here: $\Delete_t w$ and $\Delete_t \Push_t w$ are always words for the same permutation; also, if $(t,t')$ form a removable defect of $w$, then $\Delete_t w$ and $\Delete_{t'} w$ are reduced words for the same permutation.  Thus the result follows by induction on the number of steps in the bump.
\end{proof}

\begin{proposition}
\label{proposition:litle bump on dominant permutations}
Suppose $\bfa$ is a reduced word for the dominant permutation $\pi \in S_n$.  Suppose $\bfa$ is nearly reduced at $t$.  Then $\Bump^{\uparrow}_t \bfa$ has a crossing on the zeroth row.
\end{proposition}
\begin{proof}
Suppose wires $r$ and $s$ cross at position $t$.  Then the word $\Bump^{\uparrow}_t \bfa$ is a reduced word for the permutation $\pi \tau_{r,s} \tau_{i,r}$ for some $i \in I$ by Proposition~\ref{proposition:little bump properties}, part (4). We will show that $i = 0$.  This is enough, as it implies that $\pi \tau_{i,r} \tau_{i,s}$ does \emph{not} fix the point 0, so $\bfa$ must have a crossing on the zeroth row.

Suppose for a contradiction that $i>0$.  Then the defining conditions for $I(\pi \tau_{r,s},r)$ and $S(\pi \tau_{r,s},r)$  force $1 \leq i < r < s \leq n$.  Also, the same defining conditions assert that $\ell (\pi \tau_{r,s} \tau_{i,r}) = \ell(\pi \tau_{r,s}+1)$ and $\ell(\pi) = \ell(\pi \tau_{r,s}) + 1$.  This implies that $\pi(i) < \pi(s) < \pi(r)$ - in other words, $\pi$ contains the pattern 132, so it is not dominant - a contradiction.  
\end{proof}

%=================================================================================
\section{The weight-preserving Little bump}
\label{sec:weight preserving bump}

The goal of this section is to explain how to alter the Little bump $\Bump^{\uparrow}_t$ so as to preserve $\mu(\bfa)$.  Our first task is to make a single push $\Push^{\uparrow}_t$ preserve $\mu(\bfa)$.  Of course, as stated, this goal is unattainable: pushing a crossing at height $h$ necessarily reduces its contribution to $\mu(\bfa)$ from $h$ to $h-1$.  As we shall see, the solution is to \emph{add}, formally, another word to $\Push^{\uparrow}_t \bfa$ which accounts for the missing weight.

Let $W$ be the set of all words.  We will work in the vector space $\mathbb{C}W$ of finite formal sums of words in $W$.  Extend $\mu$ linearly to this vector space, and introduce the inner product $\langle \cdot,\cdot\rangle$ which makes $W$ an orthonormal basis of $\mathbb{C}W$.

\begin{definition} Let $\bfa$ be a word of length $k$. If $1 \leq t \leq k$, define
the  \emph{push-delete} operator $\PD$ as follows:
\[\PD_t\bfa = \Push^{\uparrow}_t\bfa + \Delete_{t}\bfa \in \mathbb{C}W.\]
\end{definition}

\begin{proposition}
\label{proposition:push delete preserves FK measure}
$\mu(\bfa) = \mu(\PD_t\bfa).$
\end{proposition}

\begin{proof}  Suppose that $\bfa = \bfa'a_t\bfa''$, where $a_t$ denotes the height of the crossing in the $t$th position.  Write $\mu(\bfa'\bfa'') = X$.  Then 
\[\mu(\bfa) = a_tX = (a_t-1)X + X = 
\mu(\Push^{\uparrow}_t\bfa) + \mu(\Delete_{t}\bfa) =
\mu(\PD_t\bfa).\]
\end{proof}

Now follows a variant of the Little bump algorithm, in which we use $\PD$ instead of $\Push^{\uparrow}$.   Also, instead of returning the result of the Little bump, we return the formal sum of the intermediate stages.  We call this variant the \emph{Bump-Delete algorithm}.   

\begin{algorithm}[Bump-Delete Algorithm]
\label{algorithm:bump delete}
Input: a word $\bfa'$ which is nearly reduced at $t'$.  Define $\BD_{t'}(\bfa')$ as follows:
\begin{enumerate}
\item Initialize $\bfa\leftarrow \bfa', t \leftarrow t'$, $R=0 \in \mathbb{C}W$.
\item $R \leftarrow R +\Delete_t(\bfa)$.
\item $\bfa \leftarrow \Push^{\uparrow}_{t}\bfa$.
\item $t \leftarrow \defect_t(\bfa)$.
\item If $\bfa$ is reduced, return $R$.  Otherwise go to step 2.
\end{enumerate}
\end{algorithm}

\begin{proposition}
\label{proposition: bump delete map preserves fk measure}
Let $\bfa$ be a reduced word for the dominant permutation $\pi$.   Suppose that $\bfa$ is nearly reduced at $t$, where $t$ introduces an inversion $(r,s)$, and that $\pi \tau_{r,s}$ is also a dominant permutation.  Then all of the summands of $\BD_t\bfa$ are reduced words for $\pi \tau_{r,s}$.  Moreover,  $\mu(\BD_t\bfa) = \mu(\bfa)$.
\end{proposition}

\begin{figure}
\caption{Algorithm~\ref{algorithm:bump delete}.  Read right-to-left.  Pushed crossings are also deleted, yielding summands of $\BD_t(w)$.  Note that the final word has weight zero because of the red crossing.  Compare with Figure~\ref{fig:one step of the shuffle}.
\label{}}
\begin{center}
\includegraphics[width=\figwidth]{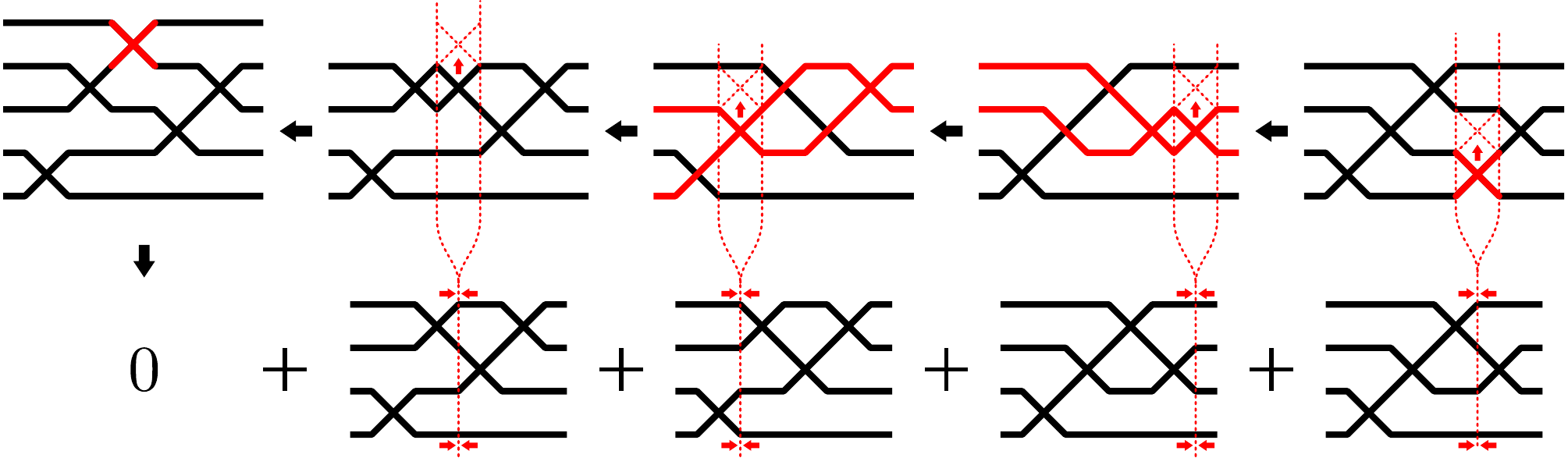}
\end{center}
\end{figure}

\begin{proof}
Proposition~\ref{proposition:little bump properties}, part 5 implies that all of the summands of $\BD_t\bfa$ are reduced words for $\pi \tau_{r,s}$.

Let $\tilde\bfa = \BD_t\bfa + \Bump^{\uparrow}_t\bfa$, and observe that $\tilde\bfa$ is the result of several iterations of $\PD$.  Thus $\mu(\tilde{\bfa}) = \mu(\bfa)$ by Proposition~\ref{proposition:push delete preserves FK measure}.   
By Proposition~\ref{proposition:litle bump on dominant permutations}, $\Bump^{\uparrow}_t \bfa$ has a crossing at height 0.  So $\mu(\Bump^{\uparrow}_t\bfa) = 0$,
and thus $\mu(\BD_t\bfa) = \mu(\bfa)$.
 \end{proof}

%====================================================================================
\section{The insert-bump map}
\label{sec:bijection}

In this section we will give a bijective proof of Theorem~\ref{thm:macdonald}.  The idea is to interpret both sides of Equation~\eqref{eqn:macdonald} as counts of the number of maximal-length paths in a certain graph.

Fix a dominant permutation $\pi$, with $k$ inversions.  Suppose the Rothe diagram of $\pi$ has shape $\lambda$, and let $T$ be an arbitrary standard Young tableau of shape $\lambda$.
$T$ gives rise to a maximal chain of Young diagrams $\emptyset = \lambda_0  <
\lambda_1 \cdots < \lambda_k = \lambda$ in Young's lattice in the usual way:
$\lambda_m$ is the shape of the subtableau of $T$ whose entries do not exceed
$m$.   

For each $m$, there is a unique dominant permutation $\pi_m$ with Young
diagram $\lambda_m$, with $\pi = \pi_k$. This can be seen inductively in $m$, by applying the fact that dominant permutations are
132 avoiding.  In fact, $\pi_m = \pi_{m-1} \tau_{i_m,j_m}$ for some pair of
wires $(i_m, j_m)$ with $i_m < j_m$, and furthermore $i_m$ is the number of the row of $T$ which contains $m$.  Let $\emptyset$ denote the
empty Young diagram, and let $\varepsilon$ denote the empty word.

\begin{definition}
\label{def:bump_delete and insert_bump}
Let $\BD_{\lambda_{m-1}}^{\lambda_m}:\mathbb{C}\red(\pi_m) \rightarrow \mathbb{C}\red(\pi_{m-1})$ be given by 
\[
\BD_{\lambda_{m-1}}^{\lambda_m}\bfa = \BD_{t_m}\bfa,
\]
where wires $r_m$ and $s_m$ cross at position $t_m$ in $\bfa$, and extending linearly.

Define the \emph{Insert-Bump} map $\IB_{\lambda_{m-1}}^{\lambda_{m}}$ to be the linear-algebraic adjoint of $\BD_{\lambda_{m-1}}^{\lambda_{m}}$ with respect to the inner product $\langle \cdot, \cdot \rangle$: that is, the unique map which satisfies $\langle \BD_{\lambda_{m-1}}^{\lambda_{m}} \bfa, \bfa' \rangle = \langle \bfa, \IB_{\lambda_{m-1}}^{\lambda_m} \bfa' \rangle$. 
\end{definition}

The terminology ``insert-bump'', and the notation $\IB$, were chosen because $\IB$ is the time-reversal of $\BD$.  Informally, $\BD$ starts from a reduced word, then bumps a crossing \emph{up}, and stops midway through the bump by deleting one of the pushed crossings. The adjoint map $\IB$, on the other hand, inserts a crossing and then bumps it \emph{down} until the resulting word is reduced.  Indeed, this is a quite general phenomenon.  Any algorithmic bijection can be inverted by reversing time (and maintaing suitable ``recording tableau'' information, typically); also, given a linear map defined in terms of an algorithmic bijection on basis vectors as we have done, one can often compute the adjoint in terms of the inverse of the bijection.  

In order to make the above remarks precise, we claim that $\IB_{\lambda_{m-1}}^{\lambda_m}$ can be performed by inserting a new crossing so as to cross wire $i_{m-1}$ with the wire above it.  This crossing can be inserted at any of the $m$ positions on this wire where one might add a crossing.  This at least makes sense: if the word $\bfa_t$ is identical to $\bfa$ but has an extra crossing inserted at position $t$ as described above, then $\bfa_t$ must be nearly reduced at $t$.  
Thus it makes sense to compute $\Bump_t^{\downarrow}\bfa_t$.  Applying parts (4) and (5) of Proposition~\ref{proposition:little bump properties}, we check that that $\Bump_t^{\downarrow}\bfa_t$ is a reduced word for $\pi_{m}$ which is necessarily nearly reduced at the position where the bump terminates.  We have thus proven the following lemma: 

\begin{lemma}
    Let $a$ be a reduced word for $\pi_{m-1}$.  For $0 \leq t \leq m-1$, let $a_t$ be the word such that $\Delete_t a_t = a$, and such that the crossing at position $t$ swaps $i_m$ with the wire above.  Then
    \[\IB_{\lambda_{m-1}}^{\lambda_{m}}\bfa = \sum_{t=1}^m \Bump_t^{\downarrow}\bfa_t.\]  
\end{lemma}

Figure~\ref{fig:one step of the shuffle} shows an example of how to calculate one of the terms in $\IB_{\lambda_{n-1}}^{\lambda_{n}}\bfa$, using the ``insert-bump'' procedure described above.  Specifically we are computing $\IB_{\lambda_{4}}^{\lambda_{5}}\bfa$, where:
\begin{itemize}
\item $\lambda_4$ is the partition $(2,1,1)$, corresponding to the dominant permutation $\pi_4 = 4213$,
\item $\lambda_5$ is the partition $(2,2,1)$, corresponding to the dominant permuation $\pi_5 = 4312$,
\item $T$ is an arbitrary standard Young tableau such that the entry 5 is in position (2,2), making $i_m = 2$,
\item $\bfa$ is the word (3,1,2,1),
\item the new crossing is inserted into position $t=1$.
\end{itemize}

It is helpful to visualize the operators in Definition~\ref{def:bump_delete and insert_bump} by drawing a ranked, directed multigraph $\Lambda_T$ as follows (see Figure~\ref{fig:macdonald shuffle graph}).  For $0 \leq m \leq k$, the vertices in the $m$th rank of $\Lambda_T$ are the reduced words for the dominant permutation $\pi_m$ (whose Rothe diagram is of shape $\lambda_m$).  All edges in $\Lambda_T$ are between consecutive ranks; $\Lambda_T$ has an edge of multiplicity $\langle \bfa, \BD^{\lambda_{m}}_{\lambda_{m-1}}\bfa'\rangle$ from $\bfa$ to $\bfa'$.

Figure~\ref{fig:macdonald shuffle graph} shows the instance of the graph $\Lambda_T$ corresponding to the standard tableau 
\[
T = \begin{matrix}
1 & 3 \\
2 & 5 \\
4
\end{matrix}
\]
The $i_m$, $1 \leq m \leq 5$, are equal to (1, 2, 1, 3, 2) respectively.  Rothe diagrams for the permutations $\pi_m$ appear at the left side of the diagram; to the right are reduced words $\pi_m$.  The dots on wire $i_m$ represent the position where the new crossing is inserted.  

  Note that the graph $\Lambda_T$ is not usually a lattice. For instance, in Figure~\ref{fig:macdonald shuffle graph}, the rightmost two words in the bottom row do not have a unique common meet, and the rightmost two words in the next-to-bottom row do not have a common join.  

\begin{definition}
\label{def:the big bijection}
Let $\mathcal{C}$ be the set of maximal-length paths in $\Lambda_T$, and let $\varepsilon$ denote the empty word.  Let $I_k$ be the set $\{1\} \times \{1,2\} \times \{1,2,3\} \times \cdots \times \{1,2,\ldots, k\}$. 
Let $\IB^T:I_k \rightarrow \mathcal{C}$ be the map which sends $(t_1, t_2, \ldots, t_k) \in I_k$ to the path in $\Lambda_T$ obtained by inserting crossings iteratively in positions $t_1, t_2, \ldots,t_k$ and bumping them down. 
\end{definition}

We are now able to prove Macdonald's formula bijectively.

\begin{proof}[Proof of Theorem~\ref{thm:macdonald}] $\IB^T$ is trivially a bijection; its inverse is obtained as follows: given a path in $\Lambda_T$ corresponding to a sequence of insertions and bumps, read off the locations $t_m$, $(1 \leq m \leq k)$ where the crossings were inserted.

Each of the maps $\IB_{\lambda_{k-1}}^{\lambda_{k}}$ preserves $\mu$, because their adjoints $\BD_{\lambda_{k-1}}^{\lambda_k}$ do.  In particular, $\mu(\bfa)$ counts the number of maximal-length paths in $\Lambda_T$ which end at $\bfa$, so the left hand side of Equation~\eqref{eqn:macdonald} is an enumeration of $\mathcal{C}$.  On the other hand, the right hand side of Equation~\eqref{eqn:macdonald}, $k!$, is equal to the size of $\mathcal{I}_k$ because all of the vertices of $\Lambda_T$ at rank $m$, $0 \leq m < k$, have outdegree $m+1$.
\end{proof}

For an example of the map $\IB^T$, please once again refer to Figure~\ref{fig:macdonald shuffle graph}.  The sequence $(1,2,2,1,3) \in I_n$ corresponds to one of the paths in $\Lambda_T$ which passes through the reduced words $\emptyset, (1), (2,1), (2,1,2), (3,2,1,2), (3,2,3,1,2)$.   

%=================================================================================
\section{The Markov growth process}
\label{sec:markov growth process}

In what follows, we shall suppress the standard tableau $T$ from our notation, and write $\IB = \IB^T$.

The map $\IB$ can be easily used to define the Markov growth process mentioned in the title of the paper.  This process can be viewed as a procedure for randomly generating reduced words $\bfa$ for a fixed dominant permutation $\pi$ of shape $\lambda$, wherein the probability of generating $\bfa$ is proportional to $\mu(\bfa)$.  The procedure for doing this is as straightforward as possible: perform simple random walk on $\Lambda_T$, starting at the empty word and ending at rank $n$.  

To see why this works, consider a ranked, directed graph $G$.  Let $\mathcal{P}_n$ be the set of $n$-step paths which start at the root $\varepsilon$ of $G$, and $\mathcal{V}_n$ be the set of vertices in $G$ of rank $n$. 

There are at least two obvious distributions on $\mathcal{V}_n$ which we must consider:

\begin{definition}
\label{definition:endpoint distributions}
Let $\mu^{\text{SRW}}_n$ be the endpoint of an $n$-step simple random walk in $G$.  Let $\mu^{\text{Uniform}}_n$ be the projection to $\mathcal{V}_n$ of the uniform distribution on $\mathcal{P}_n$.
\end{definition}

If we replace $\mathcal{V}^n$ with the vertex set of a general ranked digraph, then usually $\mu^{\text{SRW}}_n \neq \mu^{\text{Uniform}}_n$.  However, we have the following well-known fact:

\begin{lemma}
\label{lemma:uniform outdegree}
If the outdegree of every vertex $v \in \mathcal{V}_k$ is a constant $C(k)$ for all $k\leq n$, then 
$\mu^{\text{SRW}}_n = \mu^{\text{Uniform}}_n$.
\end{lemma}

\begin{proof}
Induction on $n$, the base case $n=0$ being trivial.  Suppose that $\mu^{\text{SRW}}_{k} = \mu^{\text{Uniform}}_{k}$ for some $k \geq 0$; let $v$ be a vertex in $\mathcal{V}^{k+1}$.  Write $D$ for the outdegree of all vertices at rank $k$.  Then 
\[
    \mu^{\text{SRW}}_{k+1}(v) 
    = \frac{1}{D}\sum_{u \rightarrow v}\mu^{\text{SRW}}_{k}(u)
    = \frac{1}{D}\sum_{u \rightarrow v}\mu^{\text{SRW}}_{k}(u)
    = \frac{1}{D}\sum_{u \rightarrow v}\mu^{\text{Uniform}}_{k}(u)
    = \frac{1}{D|\mathcal{V}_k|}, 
\]
which does not depend on $v$.  Thus $\mu^{\text{SRW}}_{k+1} = \mu^{\text{Uniform}}_{k+1}$, which completes the inductive step.
\end{proof}

\begin{corollary}
\label{corollary:simple random walk}
Simple random walk on $\Lambda$ produces the word $\bfa \in \red(\lambda)$ with probability proportional to $\mu(\bfa)$.
\end{corollary}

\begin{proof}
    The outdegree of $\Lambda_T$ at the $n$th rank is $n+1$, the number of places on each wire where a new crossing can be inserted into a word with $n$ crossings.  Lemma~\ref{lemma:uniform outdegree} then says that simple random walk on $\Lambda_T$ ends at $\bfa$ with probability $\mu(\bfa) / \sum_{a' \in \red(pi)} \mu(\bfa')$, where $\pi$ is the dominant permutation whose diagram is the shape of $T$. 
\end{proof}

We note that it is not in fact necessary to construct the entire (exponentially large) graph $\Lambda_T$ in order to use the simple random walk algorithm to sample from $\mu$.  Instead, we repeatedly insert crossings at a uniformly randomly chosen point on the appropriate wire, and bump them down until the word is reduced.  As such the only information which needs to be computed is the sequence of reduced words, and indeed none of these need to be retained in memory except for the most recent one.  This allows for rather efficient sampling from the Macdonald distribution.

Alexander Holroyd and the author have written an efficient sampling algorithm of the Markov growth process in python.  Figure~\ref{fig:big wiring diagram} shows the output of this implementation: it is the wiring diagram of a reduced word from the reverse permutation in $S_{600}$ (only a few wires are shown).

\begin{figure}
    \caption{Wiring diagram for a reduced word for the reverse permutation in $S_{600}$, chosen according to the Macdonald distribution.  Black dots are the locations of the crossings.  Wires $1, 50, 100, 250, 300, \ldots, 600$ are shown in red; all other wires are suppressed.
    \label{fig:big wiring diagram}}
\begin{center}
    \includegraphics[width=5.5in]{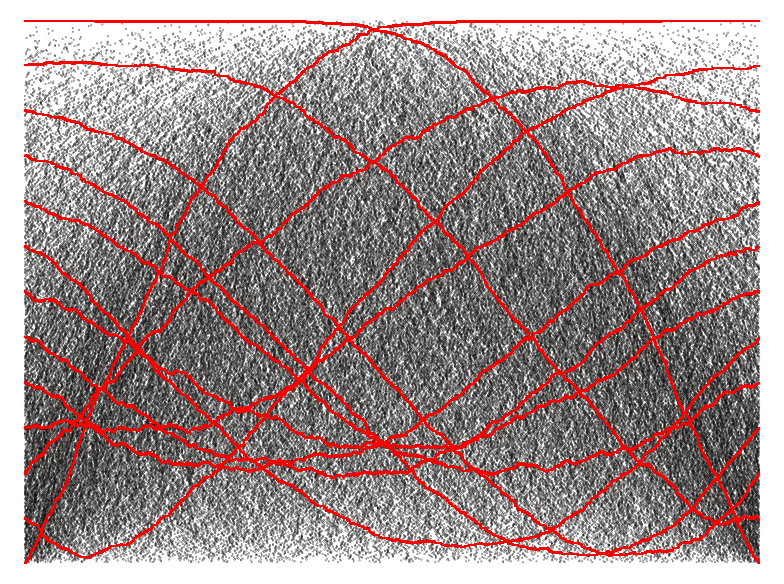}
\end{center}
\end{figure}

A few features are notable in Figure~\ref{fig:big wiring diagram}.  The trajectory of particle 1 and of particle 600 appear to be tending towards a deterministic shape.  The other trajectories vary with different runs of the algorithm, though they do also appear to be tending toward a family of smooth curves.  Likewise, the density of the swaps appears to be tending towards some smooth limiting function.

Similar phenomena occur in the case of \emph{uniformly random sorting networks}~\cite{angel-holroyd-romik-virag}, though there the authors were able to prove and conjecture much more about the limiting behavior of the process.  For instance:
\begin{itemize} 
\item Conjecturally~\cite{angel-holroyd-romik-virag} there is an explicit family of limiting curves (certain sinusoids) which the trajectories seem to approach. It is shown that the curves satisfy a H\"older condition.  Nothing similar is known for the Macdonald distribution.
\item Due to the existence of a measure-preserving action of the cyclic group, it was possible to calculate the limiting density function for the swap locations.  We do not yet have a conjectural formula for any of these quantities in our distribution.  
\item Again conjecturally~\cite{angel-holroyd-romik-virag}, the partial permutations (see Section~\ref{subsec:permutations}) matrices have a certain limiting form: the positions of the ones in their permutation matrices seem to be distributed according to the so-called \emph{Archimedes distribution}.  A much weaker form of this statement is proven in ~\cite{angel-holroyd-romik-virag} (namely, that certain triangular regions of the permutation matrices are almost surely filled with zeroes).  We were not able to formulate such conjectures here.  The ``middle'' partial permutation matrix for $\mu$-distributed sorting networks (that is, the permutation matrix for the product of the first half of the transpositions) appears at first glance to be distributed according to ``half of'' the Archimedes-distribution: the points appear to be supported on a half-ellipse-shaped region.  However, starting at approximately $n=400$ it becomes evident that the shape is \emph{not} close to a half-ellipse; the sides are somewhat flattened.  See Figure~\ref{fig:halfway}.
\end{itemize}
\begin{figure}
\caption{The permutation matrix for the product of the first $\frac{1}{2}\binom{600}{2}$ transpositions in the reduced word shown in Figure 4.  Ones are shown as dots and zeros are omitted. \label{fig:halfway}}  
\begin{center}
\includegraphics[height=4in]{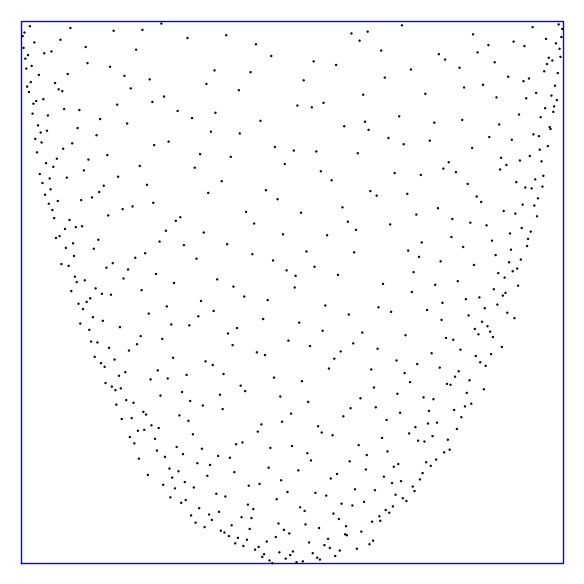}
\end{center}
\end{figure}
We would like to have even conjectural descriptions of any limiting behavour of this model.  The techniques used in~\cite{angel-holroyd-romik-virag} for the uniform distribution on reduced words do not immediately apply to the Macdonald distribution, as they rely on properties of the inverse Edelman-Greene bijection.  We do not yet have replacements for this technical tool which apply to the Macdonald distribution.

%=================================================================================
\section{The Fomin-Kirillov measure $\mu_x$}
\label{sec:fomin-kirillov}
In this section, we partially extend our results to the generalization of Macdonald's formula which was considered by Fomin-Kirillov~\cite{fomin-kirillov}.

\begin{definition} (Fomin-Kirillov weight)
    Let $\bfa = (a_1, a_2, \ldots)$ be a word and let $x \in \mathbb{N}$.  Define
\[
    \mu_x(\bfa) = \prod_t (x + a_t).
\]
\end{definition}

Given $\bfa$, a word for the dominant permutation $\pi \in S_n$, we can form the shifted word $\bfabarbar$ by replacing each $a_t \in \bfa$ with $a_t+x$ (equivalently, by adding $x$ new wires above the wiring diagram of $\bfa$).  In the notation of ~\cite{fomin-kirillov}, $\bfabarbar$ is a reduced word for the permutation $1^x \times \pi$, which fixes points 1 through $x$ and acts on the following points as $\pi$ acts on $1, \ldots, n$.  

Observe that $\mu_x(\bfa) = \mu(\bfabarbar)$, so the maps $\PD$ preserves $\mu_x(\bfa)$.  That is, the statment of Proposition~\ref{proposition:push delete preserves FK measure} holds even when $\mu$ is replaced by $\mu_x$.  However, Proposition \ref{proposition: bump delete map preserves fk measure} fails to hold, for the following reason: once a crossing is bumped to position 0, its weight under $\mu_x$ is not 0, but rather $x$.  

One way to fix this problem is to modify the way $\BD$ acts on $\bfabarbar$: after the bump terminates, apply $\PD$ $x$ more times, until the last crossing moved \emph{is} on wire zero.  Then define a new map $\BD_{t}^{x}:\red(1^x \times S_n) \rightarrow \red(1^x \times S_n)$, as follows:
\begin{definition}
Let $\bfa$ be a reduced word for a dominant permutation, which is nearly reduced at $t$. Suppose $\Bump_t^{\uparrow}\bfabarbar$ terminates after pushing the crossing in location $t'$.  Let $\BD_t^x(\bfabarbar) = \BD_t(\bfabarbar) + \left(\PD_{t'}\right)^x \Bump_t^{\uparrow}\bfabarbar$.
\end{definition}
It is now immediate from Proposition~\ref{proposition: bump delete map preserves fk measure} that $\BD_t^x$ preserves $\mu_0(\bfabarbar)=\mu_x(\bfa)$.  Note that all summands in this map
are reduced words for permutations in $1^x \times S_n$, so we will abuse notation in what follows and say that $\BD_t^x$ acts on $S_n$.

We may now begin to carry out the constructions of Section~\ref{sec:bijection}, using $\BD^x$ in place of $\BD$.  Fix a standard Young tableau $T$ of shape $\lambda$, corresponding to a sequence of Young diagrams $\emptyset = \lambda_0, \lambda_1, \ldots, \lambda_k$.  We can again define maps $\BD_{\lambda_{m-1}}^{x,\lambda{m}}$ in terms of $\BD_t^x$; we may define the insert-bump map $\IB^T$ in terms of the adjoints of the $\BD_t^x$, and finally we may associate a graph $\Lambda^x_T$ as before.  

We have the following characterization of $\Lambda^x_T$:

\begin{proposition}
\label{prop:lambda x}
The multiset of edges of $\Lambda^x_T$ coincides with that of $\Lambda_T$.  The multiplicities of the edges are the same, with the following exception: each edges arising from inserting a crossing at the \emph{top position} of $\bfa$ corresponds to $x+1$ edges in $\Lambda^x_T$. $\hfill \square$
\end{proposition}

\begin{proof}
The only new terms in $\BD^x_t$ which were not present in $\BD_t$ are the multiplicity-$x$ reduced words which arise from the action of $\left(\PD_{t'}\right)^x \Bump_t^{\uparrow}\bfabarbar$: these are $x$ copies of the same word, which can be obtained by deleting a crossing from wire $x$ of $\bfabarbar$.  Reversing time, inserting a crossing atop any of the first $x$ wires, and bumping it down repeatedly will give rise to such a crossing.
\end{proof}

If the edges are counted with the multiplicities in Proposition~\ref{prop:lambda x}, it is still the case that the Fomin-Kirillov weight of $\bfa$ is equal to the number of directed paths from the empty word to $\bfa$ in $\Lambda^x_T$.

Unfortunately, $\Lambda^x_T$ no longer has constant outdegree at each rank, so the equivalent of Corollary~\ref{corollary:simple random walk} fails to hold.
Of course, using results of~\cite{winkler}, there does exist a Markov chain on $\Lambda^x_T$ which samples from this distribution, but it is not an obvious modification of simple random walk on $\Lambda^x_T$.  In particular, in lieu of further insight, it would be necessary to construct all of $\Lambda^x_T$ in order to draw sample from $\mu_x$.  This is very inefficient.  Likewise, a new idea is needed in order to interpret the right hand side of the Fomin-Kirillov identity in terms of this lattice.  As such, the results of this paper are not adequate to prove any of the identities in~\cite{fomin-kirillov}, nor to sample from $\mu_x$.  We will address these issues in future work.

Lastly, we make a few comments about the limiting case $x \rightarrow \infty$.  The graph $\Lambda^x_T$ can be obtained from the graph of $\Lambda_T$ by adding new edges, where the only dependence on $x$ appears in the weight of the new edges added, as remarked above.  If we allow $x$ to be a very large integer, almost all of the edges in $\Lambda^x_T$ are these ``new'' ones; if we rescale the weights on the edges so that the new edges have weight one, then the limiting object $\Lambda^{\infty}_T$ is thus the \emph{unweighted} graph which consists of only the new edges.  

It is natural to ask what the structure of this graph is.  The answer was essentially worked out by Little~\cite{little2}.  Each of the vertices in $\Lambda$ is a reduced word which corresponds, under the Edelman-Greene correspondence (or equivalently under Little's bijection) to a standard tableau.  Interpret these tableaux as insertion tableaux under the classical Robinson-Schensted correspondence.  Then there is an edge from $T$ to $T'$ in $\Lambda^{\infty}_T$ when it is possible to perform a Robinson-Schensted insertion on $T$ to obtain $T'$.

%=================================================================================

\section{Future work}
\label{sec:future work}

Our algorithm has a parameter $T$ whose role we have not considered at all.  This $T$ is a standard tableau of shape $\lambda$ and it controls the wires on which insertions are done; any choice of $T$ will yield to a bijective proof of Macdonald's identity.  It would be quite interesting to learn what role $T$ plays in the bijection.  Perhaps one can prove stronger results by choosing $T$ cleverly. 

One major extension of these results will be an extension of these results to the case of non-dominant permutations, as well as bijective proofs of the identities in Fomin-Kirillov~\cite{fomin-kirillov}; this is the subject of current investigation.

One would like to study the properties of large Macdonald-distributed reduced words for the longest permutation.  These were called ``sorting networks'' in~\cite{angel-holroyd-romik-virag}, where the uniform distribution was studied in place of the Macdonald distribution.  Even better would be a study of Fomin-Kirillov-distributed reduced words, as the Fomin-Kirillov weight interpolates between the uniform weight and the Macdonald weight.  

All of the identities mentioned in this paper have $q$-analogues.  The $q$-analogue of Macdonald's identity was conjectured by Macdonald~\cite[Equation $(6.11)_q$]{macdonald}, and proven by Fomin and Stanley~\cite{fomin-stanley} using the nilCoxeter algebra.  It would be very interesting to prove these identities bijectively and we believe that the methods established here will likely apply.

More broadly, in the future we hope to develop bijective proofs of other identites of Schubert polynomials, perhaps relying on variants of the Little bijection.

\bibliographystyle{plain}
\bibliography{macdonald}

\end{document}